\documentclass[12pt]{article}

\usepackage{amssymb,amsthm,amsmath}
\usepackage{enumerate}
\usepackage{datetime}
\newcommand{\dd}{\mathrm{d}}
\newcommand{\E}{\mathbb{E}}
\newcommand{\1}{\textbf{1}}
\newcommand{\R}{\mathbb{R}}
\newcommand{\p}[1]{\mathbb{P}\left( #1 \right)}
\newcommand{\Ent}{\mathcal{S}}
\newcommand{\vp}{\varphi}
\newcommand{\vertiii}[1]{{\left\vert\kern-0.25ex\left\vert\kern-0.25ex\left\vert #1 
    \right\vert\kern-0.25ex\right\vert\kern-0.25ex\right\vert}}
\newcommand{\scal}[2]{\left\langle #1, #2 \right\rangle}

\newtheorem{lemma}{Lemma}
\newtheorem{theorem}{Theorem}
\newtheorem{proposition}{Proposition}

\theoremstyle{remark}
\newtheorem*{remark*}{Remark}
\newtheorem{remark}{Remark}

\theoremstyle{definition}

\newtheorem{conjecture}{Conjecture}

\usepackage[paper=a4paper, left=1.2in, right=1.2in, top=1in, bottom=1in]{geometry}
\title{A reverse entropy power inequality for log-concave random vectors}

\author{Keith Ball, Piotr Nayar\thanks{Supported in part by the Institute for Mathematics and its Applications with funds provided by the National Science Foundation; supported in part by NCN grant DEC-2012/05/B/ST1/00412} \ and Tomasz Tkocz}
\date{18/09/2015}

\begin{document}

\maketitle

\begin{abstract}
We prove that the exponent of the entropy of one dimensional projections of a log-concave random vector defines a $1/5$-seminorm. We make two conjectures concerning reverse entropy power inequalities in the log-concave setting and discuss some examples.
\end{abstract}



\section{Introduction}\label{sec:intro}

One of the most significant and mathematically intriguing quantities studied in information theory is the \emph{entropy}. For a random variable $X$ with density $f$ its entropy is defined as
\begin{equation}\label{eq:defent}
\Ent(X) = \Ent(f) = -\int_\R f \ln f
\end{equation}
provided this integral exists (in the Lebesgue sense). Note that the entropy is translation invariant and $\Ent(bX) = \Ent(X) + \ln|b|$ for any nonzero $b$. If $f$ belongs to $L_p(\R)$ for some $p > 1$, then by the concavity of the logarithm and Jensen's inequality $\Ent(f) > -\infty$.
If $\E X^2 < \infty$, then comparison with the standard Gaussian density and again Jensen's inequality yields $\Ent(X) < \infty$. Particularly, the entropy of a log-concave random variable is well defined and finite. Recall that a random vector in $\R^n$ is called log-concave if it has a density of the form $e^{-\psi}$ with $\psi:\R^n\to(-\infty,+\infty]$ being a convex function.

The entropy power inequality (EPI) says that
\begin{equation}\label{eq:entpower}
e^{\frac{2}{n}\Ent(X+Y)} \geq e^{\frac{2}{n}\Ent(X)} + e^{\frac{2}{n}\Ent(Y)},
\end{equation}
for independent random vectors $X$ and $Y$ in $\R^n$ provided that all the entropies exist. Stated first by Shannon in his seminal paper \cite{Sh} and first rigorously proved by Stam in \cite{St} (see also \cite{Bl}), it is often referred to as the Shannon-Stam inequality and plays a crucial role in information theory and elsewhere (see the survey \cite{DCT}). Using the AM-GM inequality, the EPI can be \emph{linearised}: for every $\lambda \in [0,1]$ and independent random vectors $X, Y$ we have
\begin{equation}\label{eq:entpowerlin}
\Ent(\sqrt{\lambda}X + \sqrt{1-\lambda}Y) \geq \lambda\Ent(X) + (1-\lambda)\Ent(Y)
\end{equation}
provided that all the entropies exist.
This formulation is in fact equivalent to \eqref{eq:entpower} as first observed by Lieb in \cite{L}, where he also shows how to derive \eqref{eq:entpowerlin} from Young's inequality with sharp constants. Several other proofs of \eqref{eq:entpowerlin} are available, including refinements \cite{Cos}, \cite{Dem}, \cite{V}, versions for the Fisher information \cite{Car} and recent techniques of the minimum mean-square error \cite{VG}.

If $X$ and $Y$ are independent and identically distributed random variables (or vectors), inequality \eqref{eq:entpowerlin} says that the entropy of the normalised sum
\begin{equation}\label{eq:Xlambda}
X_\lambda = \sqrt{\lambda}X + \sqrt{1-\lambda}Y 
\end{equation}
is at least as big as the entropy of the summands $X$ and $Y$, $\Ent(X_\lambda) \geq \Ent(X)$. It is worth mentioning that this phenomenon has been quantified, first in \cite{CS}, which has deep consequences in probability (see the pioneering work \cite{BBN} and its sequels \cite{ABBN1, ABBN2} which establish the rate of convergence in the entropic central limit theorem and the ``second law of probability'' of the entropy growth, as well as the independent work \cite{JB}, with somewhat different methods). In the context of log-concave vectors, Ball and Nguyen in \cite{BN} establish dimension free lower bounds on $\Ent(X_{1/2}) -\Ent(X)$ and discuss connections between the entropy and major conjectures in convex geometry; for the latter see also \cite{BM2}.

In general, the EPI cannot be reversed. In \cite {BCh}, Proposition V.8, Bobkov and Christyakov find a random vector $X$ with a finite entropy such that $\Ent(X+Y) = \infty$ for every independent of $X$ random vector $Y$ with finite entropy. However, for log-concave vectors and, more generally, convex measures, Bobkov and Madiman have recently addressed the question of reversing the EPI (see \cite{BM, BM1}). They show that for any pair $X,Y$ of independent log-concave random vectors in $\R^n$, there are linear volume preserving maps $T_1, T_2:\R^n \to \R^n$ such that 
\[
e^{\frac{2}{n}\Ent(T_1(X)+T_2(Y))}\leq C(e^{ \frac{2}{n}\Ent(X)} + e^{ \frac{2}{n} \Ent(Y)}),
\]
where $C$ is some universal constant. 

The goal of this note is to further investigate in the log-concave setting some new forms of what could be called a reverse EPI. In the next section we present our results. The last section is devoted to their proofs.

\section*{Acknowledgements}

The authors would like to thank Assaf Naor for pointing out the Aoki-Rolewicz theorem as well as for fruitful discussions without which Theorem \ref{thm:kappa} would not have been discovered. They are also indebted to Mokshay Madiman for his help with tracking down several references.

\section{Main results and conjectures}\label{sec:results}

Suppose $X$ is a symmetric log-concave random vector in $\R^n$. Then any projection of $X$ on a certain direction $v \in \R^n$, that is the random variable $\scal{X}{v}$ is also log-concave. Here $\scal{\cdot}{\cdot}$ denotes the standard scalar product in $\R^n$. If we know the entropies of projections in, say two different directions, can we say anything about the entropy of projections in related directions? We make the following conjecture.

\begin{conjecture}\label{conj:norm}
Let $X$ be a symmetric log-concave random vector in $\R^n$. Then the function
\[ 
N_X(v) = \begin{cases}
e^{\Ent(\scal{v}{X})} & v \neq 0, \\
0 & v = 0
\end{cases} \]
defines a norm on $\R^n$.
\end{conjecture}

The homogeneity of $N_X$ is clear. To check the triangle inequality, we have to answer really a two-dimensional question: \emph{is it true that for a symmetric log-concave random vector $(X,Y)$ in $\R^2$ we have}
\begin{equation}
\label{eq:2dim}
e^{\Ent(X+Y)} \leq e^{\Ent(X)} + e^{\Ent(Y)}?
\end{equation}
Indeed, this applied to the vector $(\scal{u}{X},\scal{v}{X})$ which is also log-concave yields $N_X(u+v) \leq N_X(u) + N_X(v)$. Inequality \eqref{eq:2dim} can be seen as a reverse EPI, cf. \eqref{eq:entpower}. It is not too difficult to show that this inequality holds up to a multiplicative constant.

\begin{proposition}\label{prop:uptoconst}
Let $(X,Y)$ be a symmetric log-concave random vector on $\R^2$. Then 
\[
	e^{\Ent(X+Y)} \leq e\left( e^{\Ent(X)} +  e^{\Ent(Y)}  \right).
\]
\end{proposition}
\begin{proof}
The argument relies on the well-known observation that for a log-concave density $f\colon\R\longrightarrow[0,+\infty)$ its maximum and entropy are related (see for example \cite{BN} or \cite{BM2}),
\begin{equation}\label{eq:entvsmax}
-\ln\|f\|_\infty \leq \Ent(f) \leq 1-\ln\|f\|_\infty.
\end{equation}

Suppose that $w$ is an even log-concave density of $(X,Y)$. The densities of $X, Y$ and $X+Y$ equal respectively
\begin{equation}\label{eq:fgh}
f(x) = \int w(x,t)\dd t, \qquad
g(x) = \int w(t,x)\dd t, \qquad
h(x) = \int w(x-t,t)\dd t.
\end{equation}
They are even and log-concave, hence attain their maximum at zero. By the result of Ball (Busemann's theorem for symmetric log-concave measures, see \cite{Ba}), the function $
\|x\|_w = (\int w(tx)\dd t)^{-1}$
is a norm on $\R^2$. Particularly,
\begin{align*} 
\frac{1}{\|h\|_\infty} & = \frac{1}{h(0)} = \frac{1}{\int w(-t,t)\dd t} = \|e_2-e_1\|_w 
\leq \|e_1\|_w + \|e_2\|_w  \\ & = \frac{1}{\int w(t,0)\dd t} + \frac{1}{\int w(0,t)\dd t}  = \frac{1}{f(0)} + \frac{1}{g(0)}
= \frac{1}{\|f\|_\infty} + \frac{1}{\|g\|_\infty}. 
\end{align*}
Using \eqref{eq:entvsmax} twice we obtain
\[
e^{\Ent(X+Y)} \leq \frac{e}{\|h\|_\infty} \leq e\cdot\left(\frac{1}{\|f\|_\infty} + \frac{1}{\|g\|_\infty}\right) 
\leq  e\cdot\left(e^{\Ent(X)} + e^{\Ent(Y)}\right).
\]
\end{proof}

Recall that the classical result of Aoki and Rolewicz says that a $C$-quasi-norm (1-homogeneous function satisfying the triangle inequality up to a multiplicative constant $C$) is equivalent to some $\kappa$-semi-norm ($\kappa$-homogeneous function satisfying the triangle inequality) for some $\kappa$ depending only on $C$ (to be precise, it is enough to take $\kappa = \ln 2/\ln(2C)$). See for instance Lemma 1.1 and Theorem 1.2 in \cite{KPR}.
In view of Proposition \ref{prop:uptoconst}, for every symmetric log-concave random vector $X$ in $\R^n$ the function $N_X(v)^\kappa = e^{\kappa\Ent(\scal{X}{v})}$ with $\kappa = \frac{\ln 2}{1 + \ln 2}$ is equivalent to some nonnegative $\kappa$-semi-norm. Therefore, it is natural to relax Conjecture \ref{conj:norm} and ask whether there is a positive universal constant $\kappa$ such that the function $N_X^\kappa$ itself satisfies the triangle inequality for every symmetric log-concave random vector $X$ in $\R^n$. Our main result answers this question positively.

\begin{theorem}\label{thm:kappa}
There exists a universal constant $\kappa >0$ such that for a symmetric log-concave random vector $X$ in $\R^n$ and two vectors $u,v \in \R^n$ we have 
\begin{equation}\label{eq:kappaRn}
e^{\kappa \Ent(\scal{u+v}{X})} \leq e^{\kappa \Ent(\scal{u}{X})}+e^{\kappa \Ent(\scal{v}{X})}.
\end{equation}
Equivalently, for a symmetric log-concave random vector $(X,Y)$ in $\R^2$ we have
\begin{equation}\label{eq:kappaR2}
e^{\kappa \Ent(X+Y))} \leq e^{\kappa \Ent(X)}+e^{\kappa \Ent(Y)}.
\end{equation}
In fact, we can take $\kappa = 1/5$.
\end{theorem}

\begin{remark}\label{rem:optkappa}
If we take $X$ and $Y$ to be independent random variables uniformly distributed on the intervals $[0,t]$ and $[0,1]$ with $t < 1$, then \eqref{eq:kappaR2} becomes $e^{\kappa t/2} \leq 1 + t^\kappa$. Letting $t \to 0$ shows that necessarily $\kappa \leq 1$. We believe that this is the extreme case and the optimal value of $\kappa$ equals $1$.
\end{remark}

\begin{remark}\label{rem:samemarginals}
Inequality \eqref{eq:kappaR2} with $\kappa = 1$ can be easily shown for log-concave random vectors $(X,Y)$ in $\R^2$ for which one marginal has the same law as the other one rescaled, say $Y \sim tX$ for some $t > 0$. Note that the symmetry of $(X,Y)$ is not needed here. This fact in the essential case of $t=1$ was first observed in \cite{CZ}. We recall the argument in the next section. Moreover, in that paper the converse was shown as well: given a density $f$, the equality 
\[
\max\{\Ent(X+Y), \ X\sim f, Y\sim f \} = {\Ent(2X)} \]
holds if and only if $f$ is log-concave, thus characterizing log-concavity. For some bounds on $\Ent(X\pm Y)$ in higher dimensions see \cite{MK} and \cite{BM1}.  
\end{remark}

It will be much more convenient to prove Theorem \ref{thm:kappa} in an equivalent form, obtained by linearising inequality \eqref{eq:kappaR2}.

\begin{theorem}\label{thm:kappalin}
Let $(X,Y)$ be a symmetric log-concave vector in $\R^2$ and assume that $\Ent(X)=\Ent(Y)$.  Then for every $\theta \in[0,1]$ we have
\begin{equation}\label{eq:kappalin}
	\Ent(\theta X+(1-\theta) Y) \leq S(X) +\frac{1}{\kappa}\ln(\theta^\kappa + (1-\theta)^\kappa),
\end{equation}
where $\kappa>0$ is a universal constant. We can take $\kappa = 1/5$.
\end{theorem}

\begin{remark}\label{rem:kappa1}
Proving Conjecture \ref{conj:norm} is equivalent to showing the above theorem with $\kappa = 1$.
\end{remark}

Notice that in the above reverse EPI we estimate the entropy of linear combinations of summands whose joint distribution is log-concave. This is different from what would be the straightforward reverse form of the EPI \eqref{eq:entpowerlin} for independent summands with weights $\sqrt{\lambda}$ and $\sqrt{1-\lambda}$ preserving variance. Suppose that the summands $X$, $Y$ are independent and identically distributed, say with finite variance and recall \eqref{eq:Xlambda}. Then, as we mentioned in the introduction, the EPI says that the function $[0,1] \ni \lambda \to \Ent(X_\lambda)$ is minimal at $\lambda = 0$ and $\lambda = 1$. Following this logic, reversing the EPI could amount to determining the $\lambda$ for which the maximum of this function occurs. Our next result shows that the somewhat natural guess of $\lambda = 1/2$ is false in general.

\begin{proposition}\label{prop:counterex}
For each positive $\lambda_0 < \frac{1}{2(2+\sqrt{2})}$ there is a symmetric continuous random variable $X$ of finite variance for which $\Ent(X_{\lambda_0}) > \Ent(X_{1/2})$.
\end{proposition}

Nevertheless, we believe that in the log-concave setting the function $\lambda \mapsto \Ent(X_{\lambda})$ should behave nicely.

\begin{conjecture}\label{conj:Xlambda}
Let $X$ and $Y$ be independent copies of a log-concave random variable. Then the function 
\[ 
\lambda \mapsto \Ent(\sqrt{\lambda}X + \sqrt{1-\lambda}Y)
\] 
is concave on $[0,1]$.
\end{conjecture}

\section{Proofs}\label{sec:proofs}

\subsection{Theorems \ref{thm:kappa} and \ref{thm:kappalin} are equivalent}\label{sec:kappalinequiv}

To see that Theorem \ref{thm:kappalin} implies Theorem \ref{thm:kappa} let us take a symmetric log-concave random vector $(X,Y)$ in $\R^2$ and take $\theta$ such that $\Ent(X/\theta)=\Ent(Y/(1-\theta))$, that is, $\theta = e^{\Ent(X)}/(e^{\Ent(X)}+e^{\Ent(Y)}) \in [0,1]$. Applying Theorem \ref{thm:kappalin} with the vector $(X/\theta,Y/(1-\theta))$ and using the identity $\Ent(X/\theta)=\Ent(X)-\ln \theta = -\ln(e^{\Ent(X)}+e^{\Ent(Y)})$ gives 
\begin{align*}
	\Ent(X+Y) \leq S(X/\theta) + \frac{1}{\kappa} \ln \left( \frac{e^{\kappa \Ent(X)}+e^{\kappa \Ent(Y)}}{(e^{\Ent(X)}+e^{\Ent(Y)})^\kappa} \right) = \frac{1}{\kappa} \ln \left( e^{\kappa \Ent(X)}+e^{\kappa \Ent(Y)} \right),
\end{align*}
so \eqref{eq:kappaR2} follows. 

To see that Theorem \ref{thm:kappa} implies Theorem \ref{thm:kappalin}, take a log-concave vector $(X,Y)$ with $\Ent(X) = \Ent(Y)$ and apply \eqref{eq:kappaR2} to the vector $(\theta X, (1-\theta)Y)$, which yields
\begin{align*} 
\Ent(\theta X + (1-\theta)Y) &\leq \frac{1}{\kappa}\ln\left(\theta^\kappa e^{\kappa \Ent(X)} + (1-\theta)^\kappa e^{\kappa\Ent(Y)}\right) \\
&= \Ent(X) + \frac{1}{\kappa}\ln\left(\theta^\kappa + (1-\theta)^\kappa\right).
\end{align*}

\subsection{Proof of Remark \ref{rem:samemarginals}}\label{sec:samemarginals}

Let $w\colon\R^2\longrightarrow[0,+\infty)$ be the density of such a vector and let $f,g,h$ be the densities of $X, Y, X+Y$ as in \eqref{eq:fgh}. The assumption means that $f(x) = tg(tx)$. By convexity,
\begin{align*}
\Ent(X+Y) &= \inf\left\{-\int h\ln p, \ p \text{ is a probability density on $\R$}\right\}.
\end{align*}
Using Fubini's theorem and changing variables yields
\begin{align*}
-\int h\ln p &= -\iint w(x,y)\ln p(x+y) \ \dd x\dd y \\
&= -\theta(1-\theta)\iint w(\theta x,(1-\theta)y)\ln p(\theta x+(1-\theta)y) \ \dd x\dd y 
\end{align*}
for every $\theta \in (0,1)$ and a probability density $p$.
If $p$ is log-concave we get
\begin{align*} 
\Ent(X+Y) \leq &-\theta^2(1-\theta)\iint w(\theta x,(1-\theta)y)\ln p(x) \ \dd x\dd y \\
&-\theta(1-\theta)^2\iint w(\theta x,(1-\theta)y)\ln p(y) \ \dd x\dd y \\
= &-\theta^2\int f(\theta x)\ln p(x) \dd x -(1-\theta)^2\int g\big((1-\theta) y\big)\ln p(y) \dd y.
\end{align*}
Set 
\[
p(x) = \theta f(\theta x) = t\theta g(t\theta x) \]
with $\theta$ such that $t\theta = 1-\theta$. Then the last expression becomes
\[ 
\theta\Ent(X) + (1-\theta)\Ent(Y) - \theta\ln \theta - (1-\theta)\ln(1-\theta). \]
Since $\Ent(Y) = \Ent(X) + \ln t = \Ent(X) + \ln\frac{1-\theta}{\theta}$, we thus obtain 
\[ 
\Ent(X+Y) \leq \Ent(X) - \ln\theta = \Ent(X) + \ln(1+t) = \ln\left(e^{\Ent(X)}+e^{\Ent(Y)}\right). \]
\bigskip

\subsection{Proof of Theorem \ref{thm:kappalin}}\label{sec:proofkappalin}

The idea of our proof of Theorem \ref{thm:kappalin} is very simple. For small $\theta$ we bound the quantity $\Ent(\theta X + (1-\theta) Y)$ by estimating its derivative. To bound it for large $\theta$, we shall crudely apply Proposition \ref{prop:uptoconst}. The exact bound based on estimating the derivative reads as follows.

\begin{proposition}\label{prop:smalltheta}
Let $(X,Y)$ be a symmetric log-concave random vector on $\R^2$. Assume that $\Ent(X)=\Ent(Y)$ and let $0 \leq \theta \leq \frac{1}{2(1+e)}$. Then
\begin{equation}\label{eq:smalltheta}
	S(\theta X + (1-\theta)Y) \leq S(X) + 60(1+e) \theta. 
\end{equation}
\end{proposition}

The main ingredient of the proof of the above proposition is the following lemma. We postpone its proof until the next subsection.

\begin{lemma}\label{lm:derivative}
Let $w:\R^2\to\R_+$ be an even log-concave function. Define $f(x) = \int w(x,y)\dd y$ and $\gamma =  \int w(0,y) \dd y\Big/\int w(x,0) \dd x$. Then we have
\[ 
\iint \frac{-f'(x)}{f(x)}yw(x,y) \dd x \dd y \leq 30\gamma \int w. \] 
\end{lemma}

\begin{proof}[Proof of Proposition \ref{prop:smalltheta}]
For $\theta =0$ both sides of inequality \eqref{eq:smalltheta} are equal. It is therefore enough to prove that $\frac{\dd}{\dd \theta} S(\theta X + (1-\theta)Y) \leq 60(1+e)$ for $0 \leq \theta \leq \frac{1}{2(1+e)}$. Let $f_\theta$ be the density of $X_\theta = \theta X + (1-\theta)Y$. Note that $f_\theta = e^{-\vp_\theta}$, where $\vp_\theta$ is convex. Let $\frac{\dd \vp_\theta}{\dd \theta}=\Phi_\theta$ and $\frac{\dd f_\theta}{\dd \theta}=F_\theta$. Then $\Phi_\theta=-F_\theta/f_\theta$.  Using the chain rule we get
\begin{align*}
	\frac{\dd}{\dd \theta} S(\theta X + (1-\theta)Y) &= -\frac{\dd}{\dd \theta} \E  \ln f_\theta = \frac{\dd}{\dd \theta} \E \vp_\theta(X_\theta) \\
	&=  \E \Phi_\theta(X_\theta)  + \E \vp_\theta'(X_\theta)(X-Y). 
\end{align*}
Moreover,
\begin{align*}
	 \E \Phi_\theta(X_\theta) = -\E F_\theta(X_\theta)/f_\theta(X_\theta) &= -\int F_\theta(x) \dd x \\
	 &= - \frac{\dd}{\dd \theta} \int f_\theta(x) \dd x = 0.
\end{align*}
Let $Z_\theta = (X_\theta,X-Y)$ and let $w_\theta$ be the density of $Z_\theta$. Using Lemma \ref{lm:derivative} with $w=w_\theta$ gives 
\begin{align*}
	\frac{\dd}{\dd \theta} S(\theta X + (1-\theta)Y) &=  - \E\left( \frac{f'_\theta(X_\theta)}{f_\theta(X_\theta)}(X-Y) \right) \\
	&= -\int \frac{f_\theta(x)}{f_\theta(x)}y w_\theta(x,y) \dd x \dd y \leq 30 \gamma_\theta,
\end{align*}
where $\gamma_\theta = \int w_\theta(0,y) \dd y / \int w_\theta(x,0) \dd x$. It suffices to show that $\gamma_\theta \leq 2(1+e)$ for $0 \leq \theta \leq \frac{1}{2(1+e)}$.  Let $w$ be the density of $(X,Y)$. Then $w_\theta(x,y)=w(x+(1-\theta)y,x-\theta y)$. To finish the proof we again use the fact that $\|v\|_w = (\int w(tv) \dd t)^{-1}$ is a norm. Note that
\[
	\gamma_\theta = \frac{\int w_\theta(0,y) \dd y }{\int w_\theta(x,0) \dd x} = \frac{\int w((1-\theta)y,-\theta y) \dd y }{\int w(x,x) \dd x} = \frac{\|e_1+e_2\|_w }{\|(1-\theta)e_1- \theta e_2\|_w}.
\] 
Let $f(x)=\int w(x,y) \dd y$ and $g(x)=\int w(y,x) \dd y$ be the densities of real log-concave random variables $X$ and $Y$, respectively. Observe that 
by \eqref{eq:entvsmax} we have 
\[
 \|f \|_\infty^{-1} \leq e^{\Ent(X)} \leq  e \|f \|_\infty^{-1}, \qquad \|g \|_\infty^{-1} \leq e^{\Ent(Y)} \leq  e \|g \|_\infty^{-1}.  
\]
Since $\|f \|_\infty^{-1} = f(0)^{-1}=\|e_1\|_w$, $\|g \|_\infty^{-1} = g(0)^{-1}=\|e_2\|_w$ and $\Ent(X)=\Ent(Y)$, this gives $e^{-1} \leq \|e_1\| / \|e_2\| \leq e$. Thus, by the triangle inequality 
\begin{align*}
	\gamma_\theta  &\leq \frac{\|e_1\|_w + \|e_2\|_w}{(1-\theta) \|e_1\|_w - \theta \|e_2\|_w} \\
	&\leq \frac{(1+e)\|e_1\|_w}{(1-\theta)\|e_1\|_w - \theta e \|e_1\|_w } = \frac{1+e}{1-\theta(1+e)} \\
	&\leq 2(1+e).
\end{align*}
\end{proof}

\begin{proof}[Proof of Theorem \ref{thm:kappalin}] 
We can assume that $\theta \in [0,1/2]$. Using Proposition \ref{prop:uptoconst} with the vector $(\theta X, (1-\theta)Y)$ and the fact that $\Ent(X)=\Ent(Y)$ we get $\Ent(\theta X + (1-\theta)Y) \leq \Ent(X)+1$. Thus, from Proposition \ref{prop:smalltheta} we deduce that it is enough to find $\kappa>0$ such that 
\[
\min\{1,60(1+e)\theta \} \leq \kappa^{-1}\ln(\theta^\kappa + (1-\theta)^\kappa), \qquad \theta \in [0,1/2]
\]
(if $60(1+e)\theta < 1$ then $\theta < \frac{1}{2(1+e)}$ and therefore Proposition \ref{prop:smalltheta} indeed can be used in this case). By the concavity and monotonicity of the right hand side it is enough to check this inequality at $\theta_0 = (60(1+e))^{-1}$, that is, we have to verify the inequality  $e^\kappa \leq  \theta_0^\kappa + (1-\theta_0)^\kappa$. We check that this is true for $\kappa=1/5$.
\end{proof}

\subsection{Proof of Lemma \ref{lm:derivative}}\label{sec:prooflm}

We start off by establishing two simple and standard lemmas. The second one is a limiting case of the so-called Gr\"unbaum theorem, see \cite{G} and \cite{Sza}.  

\begin{lemma}\label{lm:esssup}
Let $f:\R\to\R_+$ be an even log-concave function. For $\beta > 0$ define $a_\beta$ by
\[ 
a_\beta = \sup\{x>0, \ f(x) \geq e^{-\beta}f(0)\}. \]
Then we have
\[ 
2e^{-\beta}a_\beta 
\leq \ \frac{1}{f(0)}\int f \ \leq 2(1+\beta^{-1}e^{-\beta})a_\beta .\]
\end{lemma}
\begin{proof}
Since $f$ is even and log-concave, it is maximal at zero and nonincreasing on $[0,\infty)$. Consequently, the left hand inequality immediately follows from the definition of $a_\beta$. By comparing $\ln f$ with an appropriate linear function, log-concavity also guarantees that $f(x) \leq f(0)e^{-\beta\frac{x}{a_\beta}}$ for $|x| > a_\beta$, hence
\[ 
\int f \leq 2a_\beta f(0) + 2\int_{a_\beta}^\infty f(0)e^{-\beta\frac{x}{a_\beta}} \dd x =   2a_\beta f(0) + 2f(0)\frac{a_\beta}{\beta}e^{-\beta}\]
which gives the right hand inequality.
\end{proof}

\begin{lemma}\label{lm:mean}
Let $X$ be a log-concave random variable. Let $a$ satisfy $\p{X > a} \leq e^{-1}$. Then
$\E X \leq a$.
\end{lemma}

\begin{proof}
Without loss of generality assume that $X$ is a continuous random variable and that $\p{X > a} = e^{-1}$. Moreover, the statement is translation invariant, so we can assume that $a=0$. Let $e^{-\vp}$ be the density of $X$, where $\vp$ is convex. There exists a function $\psi$ of the form
\[
	\psi(x) = \begin{cases}
	ax+b, & x \geq L \\
	+\infty, & x < L
\end{cases}
\] 
such that $\psi(0)=\vp(0)$ and $e^{-\psi}$ is the probability density of a random variable $Y$ with $\p{Y > a} = e^{-1}$. One can check, using convexity of $\vp$, that $\E X \leq \E Y$. We have $1 = \int e^{-\psi} = \frac{1}{a} e^{-(b+aL)}$ and $e^{-1}=\int_{0}^\infty e^{-\psi} = \frac{1}{a}e^{-b}$. It follows that $aL=-1$ and we have $\E X \leq \E Y = \frac{1}{a}\left( L+\frac{1}{a} \right) e^{-(b+aL)}=0$.
\end{proof}

We are ready to prove Lemma \ref{lm:derivative}.

\begin{proof}[Proof of Lemma \ref{lm:derivative}]
Without loss of generality let us assume that $w$ is strictly log-concave and $w(0) = 1$. First we derive a pointwise estimate on $w$ which will enable us to obtain good pointwise bounds on the quantity $\int yw(x,y) \dd y$, relative to $f(x)$. To this end, set unique positive parameters $a$ and $b$ to be such that
$ 
w(a,0) = e^{-1} = w(0,b)$.
Consider $l \in (0,a)$. We have
\[ 
w(-l,0) = w(l,0) \geq w(a,0)^{l/a}w(0,0)^{1-l/a} = e^{-l/a}. \]
Fix $x > 0$ and let $y > \frac{b}{a}x + b$. Let $l$ be such that the line passing through the points $(0,b)$ and $(x,y)$ intersect the $x$-axis at $(-l,0)$, that is $l = \frac{bx}{y-b}$. Note that $l \in (0,a)$. Then
\begin{align*}
e^{-1} = w(0,b) \geq w(x,y)^{b/y}w(-l,0)^{1-b/y} &\geq w(x,y)^{b/y}e^{-\frac{l}{a}(1-b/y)} \\
&= \left[w(x,y)e^{-\frac{l}{a}\frac{y}{b}\frac{y-b}{y}}\right]^{b/y},
\end{align*}
hence
\[ 
w(x,y) \leq e^{x/a-y/b}, \quad \text{for $x>0$ and $y > \frac{b}{a}x+b$}.
\]
Let $X$ be a random variable with log-concave density $y\mapsto w(x,y)/f(x)$. Let us take 
$
\beta = b + b\ln(\max\{f(0),b\})$
and 
\[
\alpha = \frac{b}{a}x - b\ln f(x) + \beta.
\]
Since $f$ is maximal at zero (as it is an even log-concave function), we check that
\[ 
\alpha \geq \frac{b}{a}x - b\ln f(0) + \beta \geq \frac{b}{a}x + b,\]
so we can use the pointwise estimate on $w$ and get
\[
\int_{\alpha}^\infty w(x,y) \dd y \leq e^{x/a}\int_{\alpha}^{\infty}e^{-y/b} \dd y = be^{x/a-\alpha/b} = \frac{b}{\max\{f(0),b\}}e^{-1}f(x) 
\leq e^{-1}f(x).
\]
This means that $\p{X > \alpha} \leq e^{-1}$, which in view of Lemma \ref{lm:mean} yields
\[ 
\frac{1}{f(x)}\int y w(x,y) \dd y = \E X \leq \alpha = \frac{b}{a}x - b\ln f(x) + \beta, \quad \text{for $x>0$}. \]

Having obtained this bound, we can easily estimate the quantity stated in the lemma. By the symmetry of $w$ we have
\[ 
\iint \frac{-f'(x)}{f(x)}yw(x,y) \dd x \dd y = 2\iint_{x>0} \frac{-f'(x)}{f(x)}yw(x,y) \dd x \dd y. \]
Since $f$ decreases on $[0,\infty)$, the factor $-f'(x)$ is nonnegative for $x > 0$, thus we can further write
\begin{align*}
\iint \frac{-f'(x)}{f(x)}yw(x,y) \dd x \dd y &\leq 2\int_0^\infty -f'(x)\left(\frac{b}{a}x - b\ln f(x) + \beta\right) \dd x \\
&= 2f(0)(-b\ln f(0) + \beta) + 2\int_0^\infty f(x)\left(\frac{b}{a}-b\frac{f'(x)}{f(x)}\right) \dd x \\
&=2f(0)b\left(1 + \ln\frac{\max\{f(0),b\}}{f(0)}\right) + \frac{b}{a}\int w + 2f(0)b.
\end{align*}
Now we only need to put the finishing touches to this expression. By Lemma \ref{lm:esssup} applied to the functions $x \mapsto w(x,0)$ and $y \mapsto w(0,y)$ we obtain
\[ 
\frac{b}{a} \leq \frac{e}{2}2(1+e^{-1})\frac{\int w(0,y) \dd y}{\int w(x,0)\dd x} = (e+1)\gamma \]
and
$b/f(0) \leq e/2$.
Estimating the logarithm yields
\[ 
1 + \ln\frac{\max\{f(0),b\}}{f(0)} \leq \frac{\max\{f(0),b\}}{f(0)} \leq \frac{e}{2}. \]
Finally, by log-concavity,
\[ 
\int w(x,y) \dd x\dd y \geq \int \sqrt{w(2x,0)w(0,2y)} \dd x \dd y = \frac{1}{4}\int \sqrt{w(x,0)}\dd x \int \sqrt{w(0,y)}\dd y \]
and
\[ 
\int w(x,0) \dd x \leq \sqrt{w(0,0)}\int\sqrt{w(x,0)}\dd x = \int \sqrt{w(x,0)} \dd x. \]
Combining these two estimates we get
\[ 
f(0) = \int w(0,y) \dd y \leq \int \sqrt{w(0,y)}\dd y \leq \frac{4\int w}{\int w(x,0)\dd x} \]
and consequently,
\[ 
f(0) b \leq \frac{e}{2} f(0) f(0)  \leq 2ef(0) \frac{\int w}{\int w(x,0) \dd x} = 2e\gamma\int w. \]
Finally,
\[ 
\iint \frac{-f'(x)}{f(x)}yw(x,y) \dd x \dd y \leq (2e^2 + 5e + 1)\gamma\int w
\]
and the assertion follows.
\end{proof}

\subsection{Proof of Proposition \ref{prop:counterex}}\label{sec:counterex}

For a real number $s$ and nonnegative numbers $\alpha \leq \beta$ we define the following trapezoidal function
\[ 
T^s_{\alpha,\beta}(x) = \begin{cases}
0 & \textrm{if $x < s$ or $x > s+\alpha + \beta$,} \\
x-s & \textrm{if $s \leq x \leq s + \alpha$,} \\
\alpha & \textrm{if $s + \alpha \leq x \leq s + \beta$,} \\
s+\alpha + \beta - x & \textrm{if $s + \beta \leq x \leq s + \alpha + \beta$.}
\end{cases} \]
The motivation is the following convolution identity: for real numbers $a, a'$ and nonnegative numbers $h, h'$ such that $h \leq h'$ we have
\begin{equation}\label{eq:conv}
\1_{[a,a+h]}\star \1_{[a',a'+h']} = T^{a+a'}_{h,h'}.
\end{equation}
It is also easy to check that
\begin{equation}\label{eq:intT}
\int_\R T^{s}_{\alpha,\beta} = \alpha\beta.
\end{equation}
We shall need one more formula: for any real number $s$ and nonnegative numbers $A, \alpha, \beta$ with $\alpha \leq \beta$ we have
\begin{equation}\label{eq:entT}
I(A,\alpha,\beta) = \int_\R AT^s_{\alpha,\beta}\ln\left(AT^s_{\alpha,\beta}\right) = A\alpha\beta\ln\left(A\alpha\right) - \frac{1}{2}A\alpha^2. 
\end{equation}

Fix $0 < a < b = a + h$. Let $X$ be a random variable with the density
\[ 
f(x) = \frac{1}{2h}\left(\1_{[-b,-a]}(x) + \1_{[a,b]}(x)\right). \]
We shall compute the density $f_\lambda$ of $X_\lambda$. Denote $u = \sqrt{\lambda}$, $v = \sqrt{1-\lambda}$ and without loss of generality assume that $\lambda \leq 1/2$. Clearly, $f_\lambda(x) =  \frac{1}{u}f\left(\frac{\cdot}{u}\right)\star \frac{1}{v}f\left(\frac{\cdot}{v}\right)(x)$, so by \eqref{eq:conv} we have 
\begin{align*}
f_\lambda(x) 
&= \bigg(\1_{u[-b,-a]}\star \1_{v[-b,-a]} +  \1_{u[a,b]}\star \1_{v[-b,-a]} \\
&\ \ \ + \1_{u[-b,-a]}\star \1_{v[a,b]} + \1_{u[a,b]}\star \1_{v[a,b]} \bigg)(x)\cdot\frac{1}{(2h)^2uv} \\
&= \bigg( \underbrace{T^{-(u+v)b}_{uh,vh}}_{T_1}(x) + \underbrace{T^{ua-vb}_{uh,vh}}_{T_2}(x) + \underbrace{T^{-ub+va}_{uh,vh}}_{T_3}(x) + \underbrace{T^{(u+v)a}_{uh,vh}}_{T_4}(x) \bigg)\cdot\frac{1}{(2h)^2uv}.
\end{align*}
This symmetric density is superposition of 4 trapezoid functions $T_1, T_2, T_3, T_4$ which are certain shifts of the same trapezoid function
$ 
T_{0} = T^0_{uh,vh}$.
The shifts may overlap depending on the value of $\lambda$. Now we shall consider two particular values of $\lambda$.

\bigskip\noindent \emph{Case 1: $\lambda = 1/2$.} Then $u=v=1/\sqrt{2}$. Notice that $T_0$ becomes a triangle looking function and $T_2 = T_3$, so we obtain
\begin{align*}
f_{1/2}(x) &= \frac{1}{2h^2}\left(T^{-b\sqrt{2}}_{h/\sqrt{2},h/\sqrt{2}} + 2T^{-h/\sqrt{2}}_{h/\sqrt{2},h/\sqrt{2}} + T^{a\sqrt{2}}_{h/\sqrt{2},h/\sqrt{2}}\right)(x).
\end{align*}
If $h/\sqrt{2} < a\sqrt{2}$ then the supports of the summands are disjoint and with the aid of identity \eqref{eq:entT} we obtain
\begin{align*}
\Ent(X_{1/2}) &= -2I\left(\frac{1}{2h^2},\frac{h}{\sqrt{2}},\frac{h}{\sqrt{2}}\right) - I\left(\frac{1}{h^2},\frac{h}{\sqrt{2}},\frac{h}{\sqrt{2}}\right) 
= \ln(2h) + \frac{1}{2}.
\end{align*}

\bigskip\noindent \emph{Case 2: small $\lambda$.} Now we choose $\lambda = \lambda_0$ so that the supports of $T_1$ and $T_2$ intersect in such a way that the down-slope of $T_1$ adds up to the up-slope of $T_2$ giving a flat piece. This happens when 
$-b(u+v) + vh = ua - bv$,
that is,
\begin{equation}\label{eq:ahcond}
\sqrt{\frac{1-\lambda_0}{\lambda_0}} = \frac{v}{u} = \frac{a+b}{h} = 2\frac{a}{h} + 1. 
\end{equation}
The earlier condition $a/h>1/2$ implies that $\lambda_0 < 1/5$. With the above choice for $\lambda$ we have
$T_1 + T_2 = T^{-b(u+v)}_{uh,2vh}$,
hence by symmetry
\[ 
f_{\lambda} =  \bigg( T^{-b(u+v)}_{uh,2vh} + T^{-ub+va}_{uh,2vh} \bigg)\cdot\frac{1}{(2h)^2uv}.\]
As long as
$ 
-ub+va > 0$,
the supports of these two trapezoid functions are disjoint. Given our choice for $\lambda$, this is equivalent to
$v/u > b/a = 1 + h/a = 1 + 2/(v/u-1)$, or putting $v/u = \sqrt{1/\lambda_0-1}$, to $\lambda_0 < \frac{1}{2(2+\sqrt{2})}$.
Then also $\lambda_0 < 1/5$ and we get
\begin{align*}
\Ent(X_\lambda) &= -2I\left(\frac{1}{(2h)^2uv},uh,2vh\right) = \ln(4vh) + \frac{u}{4v} = \ln(4h\sqrt{1-\lambda_0}) + \frac{1}{4}\sqrt{\frac{\lambda_0}{1-\lambda_0}}.
\end{align*}
We have
\[ 
\Ent(X_{\lambda_0}) - \Ent(X_{1/2}) = \ln 2 - \frac{1}{2} + \ln\sqrt{1-\lambda_0} + \frac{1}{4}\sqrt{\frac{\lambda_0}{1-\lambda_0}}.\]
We check that the right hand side is positive for $\lambda_0 < \frac{1}{2(2+\sqrt{2})}$. Therefore, we have shown that for each such $\lambda_0$ there is a choice for the parameters $a$ and $h$ (given by \eqref{eq:ahcond}), and hence a random variable $X$, for which $\Ent(X_{\lambda_0}) > \Ent(X_{1/2})$.

\noindent Keith Ball$^\star$, \texttt{k.m.ball@warwick.ac.uk}

\vspace{1em}

\noindent Piotr Nayar$^\dagger$, \texttt{nayar@mimuw.edu.pl}

\vspace{1em}

\noindent Tomasz Tkocz$^{\star}$, \texttt{t.tkocz@warwick.ac.uk}

\vspace{1em}

\noindent $^\star$Mathematics Institute, University of Warwick, \\
\noindent Coventry CV4 7AL, \\
\noindent UK

\vspace{1em}

\noindent $^\dagger$Institute of Mathematics \& Applications, \\
Minneapolis MN 55455 \\
United States

\end{document}